\documentclass[11pt]{amsart}
\usepackage[utf8]{inputenc}
\usepackage{amsmath, amsthm, amssymb, bbm, bm, amsfonts}
\usepackage[colorinlistoftodos]{todonotes}
\usepackage{caption}
\usepackage[english]{babel}
\usepackage{hyperref}

\newcommand{\genleg}[2]{
    \genfrac{(}{)}{}{}{#1}{#2}
}
\newcommand{\Mod}[1]{\ (\mathrm{mod}\ #1)}
\newcommand{\Li}{\text{Li}}

\def \N {{\mathbb N}}

\def \Z {{\mathbb Z}}
\def \Q {{\mathbb Q}}
\newtheorem{theorem}{Theorem}[section]

\newtheorem{lemma}[theorem]{Lemma}
\newtheorem{remark}[theorem]{Remark}

\newtheorem{conjecture}[theorem]{Conjecture}

\newtheorem{cor}[theorem]{Corollary}
\numberwithin{equation}{section}
 
\title[Artin Twin Primes]{Artin Twin Primes}
\author{Magdal\'ena Tinkov\'a, Ezra Waxman, Mikul\'a\v{s} Zindulka}
\date{\today}

\address{Charles University, Faculty of Mathematics and Physics, Department of Algebra,
Sokolovsk\'{a} 83, 18600 Praha 8, Czech Republic}
\email{tinkova.magdalena@gmail.com}

\address{Technische Universität Dresden, Fakultät Mathematik, Institut für Algebra, Zellescher Weg 12-14, 01062 Dresden, Germany}
\email{ezrawaxman@gmail.com}

\address{Charles University, Faculty of Mathematics and Physics, Department of Mathematical Analysis,
Sokolovsk\'{a} 83, 18600 Praha 8, Czech Republic}
\email{zinmik2@gmail.com}

\begin{document}

\maketitle

\begin{abstract}
We say that a prime number $p$ is an \textit{Artin prime} for $g$ if $g$ mod $p$ generates the group $(\mathbb{Z}/p\mathbb{Z})^{\times}$.  For appropriately chosen integers $d$ and $g$, we present a conjecture for the asymptotic number $\pi_{d,g}(x)$ of primes $p \leq x$ such that both $p$ and $p+d$ are Artin primes for $g$.  In particular, we identify a class of pairs $(d,g)$ for which $\pi_{d,g}(x) =0$.  Our results suggest that the distribution of Artin prime pairs, amongst the ordinary prime pairs, is largely governed by a Poisson binomial distribution.
\end{abstract}

\section{Introduction}
\subsection{Background}
Fix an odd prime $p \in \mathbb{N}$, and let $(\Z/p\Z)^{\times}$ denote the finite cyclic group of order $p-1$.  We say that $g \in \Z$ is a \textit{primitive root} mod $p$ (denoted $ord_{p}(g)=p-1$) if $g$ mod $p$ generates the group $(\Z/p\Z)^{\times}$.  Let $\mathcal{P}$ denote the set of odd primes, and let $\mathcal{P}_{g} \subseteq \mathcal{P}$ denote the subset of primes $p$ for which $ord_{p}(g)=p-1$.  If $p \in \mathcal{P}_{g}$, we say that $p$ is an \textit{Artin prime} (for the root $g$).  \textit{Artin's primitive root conjecture} states that if $g$ is neither a perfect square nor $-1$, then $\mathcal{P}_{g}$ contains infinitely many primes.\\
\\
Fix an integer $g \neq -1$ such that $g$ is not a perfect square.  Let $h \in \mathbb{N}$ denote the largest integer such that $g = b^{h}$ for some $b \in \Z$, and let $\Delta$ denote the discriminant of $\Q(\sqrt{g})$.  Assuming the \textit{Generalized Riemann Hypothesis} (GRH) for Dedekind zeta functions, Hooley \cite{Ho} demonstrated that the \textit{natural density} of $\mathcal{P}_{g}$ in $\mathcal{P}$ is equal to $c_{g} \cdot A(h)$, i.e. that
\begin{equation}\label{Hooley}
\lim_{x \rightarrow \infty}\frac{\# \{p \in \mathcal{P}_{g}: p \leq x\}}{\# \{p \in \mathcal{P}: p \leq x\}} = c_{g}\cdot A(h),
\end{equation}
where
\[A(h):=\prod_{\substack{p\nmid h\\ p \textnormal{ prime}}}\left(1-\frac{1}{p(p-1)}\right)\prod_{\substack{p|h \\ p \textnormal{ prime}}}\left(1-\frac{1}{p-1}\right)\]
and 
\begin{equation}\label{classical correction factor}
c_{g}:=1-\mu(|\Delta|)\prod_{\substack{q|\Delta\\ q|h}}\frac{1}{q-2}\prod_{\substack{q|\Delta\\ q\nmid h}}\frac{1}{q^2-q-1}
\end{equation}
is a correction factor.  In particular, $c_{g}=1$ whenever $\Delta \not \equiv 1(4)$.\\
\\
The correction factor $c_{g}$ has an interesting history \cite{Stevenhagen}.  Initially, E. Artin (1927) actually claimed that the natural density of $\mathcal{P}_{g}\subseteq \mathcal{P}$ is equal to $A(h)$ for \textit{every} integer $g \neq -1$ that is not a square.  Such a model is based on the assumption that any two field extensions of the form $K_{q}=\Q(\zeta_{q},\sqrt[q]{g})$, $q$ prime, are linearly disjoint over $\Q$.  In 1957, Derrick and Emma Lehmer brought to Artin's attention the need for an additional correction factor, by noting numerical deviations from Artin's predicted asymptotic in the case $g=5$.  In a reply letter to the Lehmers, Artin suggested a modified heuristic that could correct for the `entanglement' of the corresponding splitting fields.  The adjusted heuristic could better explain the newly observed numerical data, and Artin concluded his letter by noting that he had been ``careless but the machine caught up with me."  In 1965, Lang and Tate eventually listed the full correction factor, $c_{g}$, in their preface to Artin's collected works, and this corrected conjecture was then proved by Hooley two years later, under GRH.  For a derivation of $c_{g}$ (as well as more general Artin-type correction factors) in terms of character sums, see \cite{LSM}.
\subsection{Modeling Artin Primes}
The group $(\Z/p\Z)^{\times}$ has $p-1$ elements and $\varphi(p-1)$ generators, where $\varphi$ denotes Euler's totient function.  The fraction $\varphi(p-1)/p-1$ may thus be viewed as the ``probability" that $g$ is a primitive root mod $p$, and the quantity 
\[\mathcal{P}_{g}(x):=\# \{p \in \mathcal{P}_{g}: p \leq x\}\]
naively modeled by the sum
\[W_{g}(x)=\sum_{\substack{p \leq x\\ p \textnormal{ prime}}}\frac{\varphi(p-1)}{p-1}.\]
Indeed,
\[W_{g}(x) \sim A(1) \textnormal{Li}(x),\]
i.e. the above heuristic yields the correct asymptotic whenever $h=1$ and $c_{g}=1$.\\
\\
To improve upon this model, we make two simple observations.  First, a necessary condition for $p$ to lie in $\mathcal{P}_{g}$ is for $(p-1,h)=1$.  Indeed, if $(p-1,h)>1$, then 
\begin{equation}
g^{\frac{p-1}{(p-1,h)}}\equiv \left(b^{\frac{h}{(p-1,h)}}\right)^{p-1}\equiv 1 \Mod{p},
\end{equation}
and therefore $p \not \in \mathcal{P}_{g}$.  This corresponds to the necessary adjustment of the main term, $A(h)$.  Second, a further necessary condition for $p$ to lie in $\mathcal{P}_{g}$ is for $g$ to be a non-square mod $p$, i.e. for $(g/p)=-1$, where $(g/p)$ is the Legendre symbol.  This corresponds to the additional correction factor, $c_{g}$.\\
\\
If the above two conditions are met, then $g$ lies in one of $(p-1)/2$ specified residue classes in $(\Z/p\Z)^{\times}$, of which $\varphi(p-1)$ are primitive roots. Such considerations motivated Moree \cite{Mo1, Mo2} to propose the following general model to mimic the properties of Artin primes.  Fix an integer $g \neq -1$ such that $g$ is not a perfect square.  For $p \in \mathcal{P}$, define
\begin{align}\label{wp}
w_{g}(p) &:= \left\{
\begin{array}{l l l}
2 \frac{\varphi(p-1)}{p-1}& \text{ if } \left(\frac{g}{p}\right) = -1 \textnormal{ and } (p-1,h)=1\\
0 & \textnormal{ otherwise}.
\end{array} \right.
\end{align}
Suppose $S \subseteq \mathcal{P}$ has a non-zero natural density, and let 
\[N_{S}(x):=\{p \leq x: p \in S\}\]
denote the set of primes $p \leq x$ contained in $S$. To each $p_{i}\in S$, we assign a Bernoulli distributed random variable $E_{p_{i}}$, such that $E_{p_{i}} = 1$ occurs with probability $w_{g}(p_{i})$, and $E_{p_{i}} = 0$ occurs with probability $1-w_{g}(p_{i})$.  For fixed $x$, we then consider the random variable
\[X_{g,N} = \sum_{\substack{p_{i} \leq x \\ p_{i} \in S}}E_{p_{i}}\]
representing the number of successes amongst a sequence of $N:=|N_{S}(x)|$ trials, $\{E_{p_{i}}\}_{i=1}^{N}$. As $X_{g,N}$ is a finite sum of independent Bernoulli distributed random variables, it has a \textit{Poisson binomial distribution}, with mean value
\[\mu_{g,N}=\sum_{\substack{p \leq x \\ p \in S}}w_{g}(p)\]
and variance
\[\sigma^{2}_{g,N}=\sum_{\substack{p \leq x \\ p \in S}}(1-w_{g}(p))w_{g}(p).\]

If the elements in $\mathcal{P}_{g}$ are well distributed across admissible residue classes, one expects
\[\#\{p: p \leq x, p \in \mathcal{P}_{g} \cap S\}\]
to look like a ``typical" instance of $X_{g,N}$.
\subsection{Artin Primes in Arithmetic Progressions}
As an example, set 
\[S=\{p:p \equiv a\Mod f\},\] where $(a,f)=1$, and let
\[\mathcal{P}_{g}(x;f,a):=\#\{p: p\leq x, p \equiv a \Mod f, \text{ and } p \in \mathcal{P}_{g} \}.\]
For simplicity sake, we moreover assume that $h=1$, i.e. that $g \in \mathcal{H}_{1}$, where
\[\mathcal{H}_{1}:= \{g \in \Z: g \neq -1, g \textnormal { is not a perfect power}\}\subset \Z.\]
 Under GRH, 
\begin{equation}\label{prob model}
\mathcal{P}_{g}(x;f,a) = \textbf{A}(a,f)\textbf{c}_{g}(a,f)\frac{x}{\log x}+O\left(\frac{x \log \log x}{\log 2x}\right)
\end{equation}
where
\begin{equation}
\textbf{A}(a,f):= \frac{1}{\varphi(f)}\prod_{p|(a-1,f)}\left(1-\frac{1}{p}\right)\prod_{p \nmid f}\left(1-\frac{1}{p(p-1)}\right)
\end{equation}
and $\textbf{c}_{g}(a,f)$ is a correction factor that takes into account quadratic obstructions \cite{Le, Stevenhagen}.  In particular, $\textbf{c}_{g}(a,f)=0$ if and only if $\Delta|f$ and $(\Delta/a)=1$, where $(a/b)$ denotes the Kronecker symbol.\\
\\
The probabilistic model presented above suggests that in the limit as $x \rightarrow \infty$,
\begin{equation}\label{prob model}
\mathcal{P}_{g}(x;f,a) \sim W_{g,S}(x):=\sum_{\substack{p \leq x \\ p \equiv a(\text{mod } f)}}w_{g}(p).
\end{equation}
Indeed, Moree \cite{Mo1, Mo2} showed that
\[W_{g,S}(x) \sim \textbf{A}(a,f)\textbf{c}_{g}(a,f)\frac{x}{\log x},\]
i.e. that the model provides an exact asymptotic heuristic for $\mathcal{P}_{g}(x;f,a)$, under GRH.  Variants of Moree's model moreover provide exact asymptotic heuristics for a wide class of problems related to Artin's conjecture, such as an asymptotic count for the number of primes $p \leq x$ for which $ord_{p}(g)$ is even \cite{Ciolan, Moree3} (unconditional) or for which $g$ is a near primitive root \cite{Moree4, Moree5} (conditional on GRH).  Such models are related to observations suggesting that the distribution of primitive roots across admissible residue classes modulo a large fixed prime $p \in \mathcal{P}$ is largely governed by Poisson processes \cite{Cobeli, Rudnick}.
\subsection{A Hardy-Littlewood Conjecture for Artin Primes}
Let $\bm{n}=(n_{1},\dots,n_{k})$ denote an integer $k$-tuple, i.e. an ordered set of $k$ integers, and write $\bm{n} \leq x$ if $n_{i} \leq x$ for all $1 \leq i \leq k$.\\
\\
Fix $\bm{d}=(d_{1},\dots,d_{k}) \in (\Z_{\geq 0})^{k}$ to be a $k$-tuple of even integers.  We say that $\bm{d}$ is \textit{admissible} if it does not include the complete residue system of any prime $p$.  Let
\[\bm{D}:=\{\bm{n} \in \mathbb{N}^{k}: n_{i}= n+d_{i} \text{ for all } 1 \leq i \leq k, \text{ for some odd }n \in \mathbb{N} \}\]
denote the set of (odd) integer $k$-tuples with spacing $\bm{d}$.  Define the \textit{prime $k$-tuples} (of spacing $\bm{d}$) to be the collection
\[\pi_{\bm{d}}:=\{\bm{p} \in \bm{D} : p_{i}=n+d_{i} \in \mathcal{P} \text{ for all } 1 \leq i \leq k\}.\]
In 1922, Hardy and Littlewood \cite{HL} conjectured that in the limit as $x \rightarrow \infty$,
\[\#\{\bm{p}: \bm{p} \leq x \text{ and }\bm{p} \in \pi_{\bm{d}}\} = \mathfrak{S}(\bm{d})\cdot \frac{x}{(\log x)^{k}}\left(1+o(1)\right),\]
where $\mathfrak{S}(\bm{d})$ is the \textit{singular series}
\[\mathfrak{S}(\bm{d}):= \prod_{p}\frac{1-\omega_{\bm{d}}(p)/p}{(1-1/p)^{k}},\] 
and where $\omega_{\bm{d}}(p)$ counts the number of distinct residue classes amongst $d_{1},\dots, d_{k}$ modulo $p$.  Note that $\mathfrak{S}(\bm{d})>0$ for any admissible $\bm{d}$.\\
\\
We wish to study subsets of $\pi_{\bm{d}}$ with fixed primitive roots.  Fix $\bm{d}=(d_{1},\dots ,d_{k})$ to be an admissible $k$-tuple, and $\bm{g}=(g_{1},\dots,g_{k}) \in (\mathbb{Z} \setminus \{-1\})^{k}$ to be an integer $k$-tuple such that $g_{i}$ is not a perfect square for any $1 \leq i \leq k$.  We define the set of \textit{Artin prime $k$-tuples} (of spacing $\bm{d}$ with respect to $\bm{g}$) to be
\[\pi_{\bm{d},\bm{g}}:=\{\bm{p} \in \bm{D}: p_{i} =n+d_{i} \in \mathcal{P}_{g_{i}} \text{ for all } 1 \leq i \leq k\} \subseteq \pi_{\bm{d}}.\]
To model this set quantitatively, we employ a variation of the heuristic function $w_{g}(p)$ defined above.  Specifically, for $\bm{p} \in \pi_{\bm{d}}$, define
\[w_{\bm{g}}(\bm{p}) := w_{g_{1}}(p_{1})\cdots w_{g_{k}}(p_{k}).\]
To each $\bm{p}_{i} \in \pi_{\bm{d}}$, we assign a Bernoulli distributed random variable $E_{\bm{p}_{i}}$, such that $E_{\bm{p}_{i}}=1$ with probability $w_{\bm{g}}(\bm{p}_{i})$.  We then model $\pi_{\bm{d},\bm{g}}$ by a Poisson binomial distribution indexed by the events $\{E_{\bm{p}_{i}}\}_{\bm{p}_{i} \in \pi_{\bm{d}}}$.  Specifically, define
\[\pi_{\bm{d},\bm{g}}(x):=\#\{\bm{p}: \bm{p} \leq x \textnormal{ and }\bm{p} \in \pi_{\bm{d},\bm{g}}\}.\]
If $\pi_{\bm{d},\bm{g}}(x)$ is unbounded, we expect that in the limit as $x \rightarrow \infty$,
\begin{equation}\label{Full Hardy Littlewood Artin Primes}
\pi_{\bm{d},\bm{g}}(x) \sim W_{\bm{g},\bm{d}}(x):=\sum_{\substack{\bm{p} \leq x \\ \bm{p} \in \pi_{\bm{d}}}}w_{\bm{g}}(\bm{p}).
\end{equation}
\subsection{Artin Prime Pairs}
In this paper, we use the heuristic model suggested in (\ref{Full Hardy Littlewood Artin Primes}) to study the asymptotic growth of \textit{Artin Prime Pairs}, i.e. $\pi_{\bm{d},\bm{g}}$ when $k=2$, $\bm{d}=(0,d)$ and $\bm{g}=(g,g)$.  Specializing our notation, we let
\[\pi_{d}:=\{p: p,p+d \in \mathcal{P}\}\]
denote the set of prime gaps of size $d \in 2\N$, and
\[\pi_{d,g}:=\{p: p,p+d \in \mathcal{P}_{g}\}\]
denote the set of $p \in \pi_{d}$ such that $g$ is a primitive root modulo both $p$ and $p+d$.  Under the assumption of GRH, Pollack \cite{BakerPollack, Pollack} showed the existence of some $C > 0$ such that $\pi_{d,g}$ is an infinite set for some $d \leq C$.\\
\\
Fix $d \in 2\N$.  Henceforth, we will moreover assume $g \in\mathcal{H}_{1}$.  For $p \in \pi_{d}$, define 
\[\Delta_{d,g}(p) := \left\{
\begin{array}{l l}
1 & \text{ if } \genleg{g}{p}=\genleg{g}{p+d}=-1, \\
0 & \text{ otherwise}.
\end{array} \right.\]
Note that since $p \in \mathcal{P}_{g} \Rightarrow \genleg{g}{p} = -1$, it follows that $\Delta_{d,g}(p)=0 \Rightarrow p \not \in \pi_{d,g}.$  In particular, if $\Delta_{g,d}(p)=0$ for all $p \in \pi_{d}$, then $\pi_{d,g}$ is an empty set.  Upon identifying a set of pairs $(g,d)$ such that $\Delta_{g,d}(p)=0$ for all $p \in \pi_{d}$, we obtain the following theorem:
\begin{theorem} \label{Finite Theorem}
Fix $d \in 2\N$ and $g \in \mathcal{H}_{1}$, and write $g = g_{0}g_{s}^2$, where $g_{0}$ is square-free. Then $\pi_{d,g}$ is empty if the pair $(d,g)$ satisfies any of the following conditions:
\begin{enumerate}
\item $g_0 = 5 $ and $ d \equiv 2, 3 \pmod{5} $,
\item $g_0| d $, and either
    \begin{enumerate}
 \item $ d \equiv 2 \pmod{4} $ and $g_0 \equiv 3 \pmod{4} $,
\item $ d \equiv 4 \pmod{8} $ and $g_0 \equiv 2 \pmod{4} $,
    \end{enumerate}
\item $g_0 \nmid d$, $g_0 \mid 3d$, and either
    \begin{enumerate}
        \item $g_0 \equiv 1 \pmod{4} $
    \item $g_0 \equiv 3 \pmod{4}$ and $ d \equiv 0 \pmod{4} $,
    \item $g_0 \equiv 2 \pmod{4} $ and $ d \equiv 0 \pmod{8} $.
    \end{enumerate}
\end{enumerate}
\end{theorem}
Conversely, if $d \in 2\N$ and $g \in \mathcal{H}_{1}$ do not satisfy any of the conditions identified in Theorem \ref{Finite Theorem}, we conjecture $\pi_{d,g}$ to be an infinite set (and moreover $\pi_{d,g} \subseteq \pi_{d}$ to have a non-zero natural density amongst the twin-prime pairs).  Let
\[\pi_{d}(x):= \#\{p \in \pi_{d}:p,p+d \leq x\}\]
and
\[\pi_{d,g}(x) :=\#\{p \in \pi_{d,g}: p, p+d \leq x\}.\]
The main aim of this paper is to suggest $-$ and provide support for $-$ the following conjecture:
\begin{conjecture}$($\textbf{Artin Two-Tuple Conjecture}$)$:\label{main conjecture}
Fix $d \in 2\N$ and $g \in \mathcal{H}_{1}$.  Then 
\[\lim_{x \rightarrow \infty}\frac{\pi_{d,g}(x)}{\pi_{d}(x)}= \textbf{c}_{g}(d)\mathbf{A}(d),\]
where 
\begin{equation}
\textbf{A}(d):= \prod_{q|d}\left(1+\frac{1}{q^2(q-1)}-\frac{2}{q(q-1)}\right)\prod_{q \mid d\pm 1}\left(1-\frac{1}{q(q-2)}\right)\prod_{q \nmid d(d\pm 1)}\left(1-\frac{2}{q(q-2)}\right),
\end{equation}
and $\textbf{c}_{g}(d)$  is as in 
\textnormal{(\ref{correction factor g=2})}, \textnormal{(\ref{correction factor odd})}, or \textnormal{(\ref{correction factor even})}.  In particular, $\textbf{c}_{g}(d)=0$ if and only if the pair $(d,g)$ satisfies any of the conditions identified in Theorem \ref{Finite Theorem}.
\end{conjecture}
$\mathbf{A}(d)$ may be viewed as a `linear' constant which emerges upon ignoring the contribution of all non-principal characters in the computation of $\pi_{d,g}$.  $\textbf{c}_{g}(d)$ is a correction factor that includes the contribution from quadratic characters.  Conjecture \ref{main conjecture} is derived from the heuristic assumption that the combined contribution of all additional higher-order characters is asymptotically negligible (Conjecture \ref{E Conjecture}).
\\
\\
When $|g|=2$ the correction factor is particularly simple, namely
\begin{equation}\label{correction factor g=2}
\textbf{c}_{g}(d) := \left\{
\begin{array}{l l l l}
2 & \text{ if } d \equiv 0(8)\\
1 & \text{ if } d \equiv 2,6(8)\\
0 & \text{ if } d \equiv 4(8).
\end{array} \right.
\end{equation}
Pollack \cite{Pollack} states a conjecture for $\pi_{2,2}(x)$, i.e. for the particular case $g=2$ and $d=2$, though he does not provide any theoretical justification.  Zoeteman \cite{Zo} extends this to a conjecture for $\pi_{d,2}(x)$ ($d \not \equiv 4(8)$), and outlines a theoretical approach based upon the circle method of Ramanujan, Hardy, and Littlewood.  Our work provides independent corroboration, further generalization, and additional theoretical justification, for these earlier conjectures.  In particular, our results are obtained by quite different methods than those outlined by Zoeteman, who employs an alternative heuristic for counting Artin prime solutions to general Diophantine equations, implicit in the work of Frei, Koymans, and Sofos \cite{Frei}.\\
\\
A strong form of the Hardy-Littlewood conjecture asserts that
\[\lim_{x \rightarrow \infty}\pi_{d}(x)= \mathfrak{S}(d)\cdot \Li_2(x)+ O(x^{1/2+\epsilon}),\]
where
\begin{equation}\label{twin prime constant}
\mathfrak{S}(d) := 2\prod_{\substack{p \mid d\\p>2}}\left(\frac{p-1}{p-2}\right)\prod_{p > 2}\left(1-\frac{1}{(p-1)^2}\right)
\end{equation}
and
\begin{equation}\label{Li 2}
\textnormal{Li}_{2}(x) := \int_2^x\frac{1}{(\log t)^{2}}dt.
\end{equation}
We expect a similar square-root cancellation to hold for Artin twin primes, namely that
\begin{equation}\label{conjecture in x}
\pi_{d,g}(x) = \textbf{c}_{g}(d)\mathbf{A}(d)\mathfrak{S}(d)\Li_2(x)+ O(x^{1/2+\epsilon}).
\end{equation}
Numerical data for Conjecture \ref{main conjecture} is provided in Figure 1 and Tables 1$-$3.
\begin{figure}[h]\label{one billion}
\centering
\includegraphics[scale=0.9]{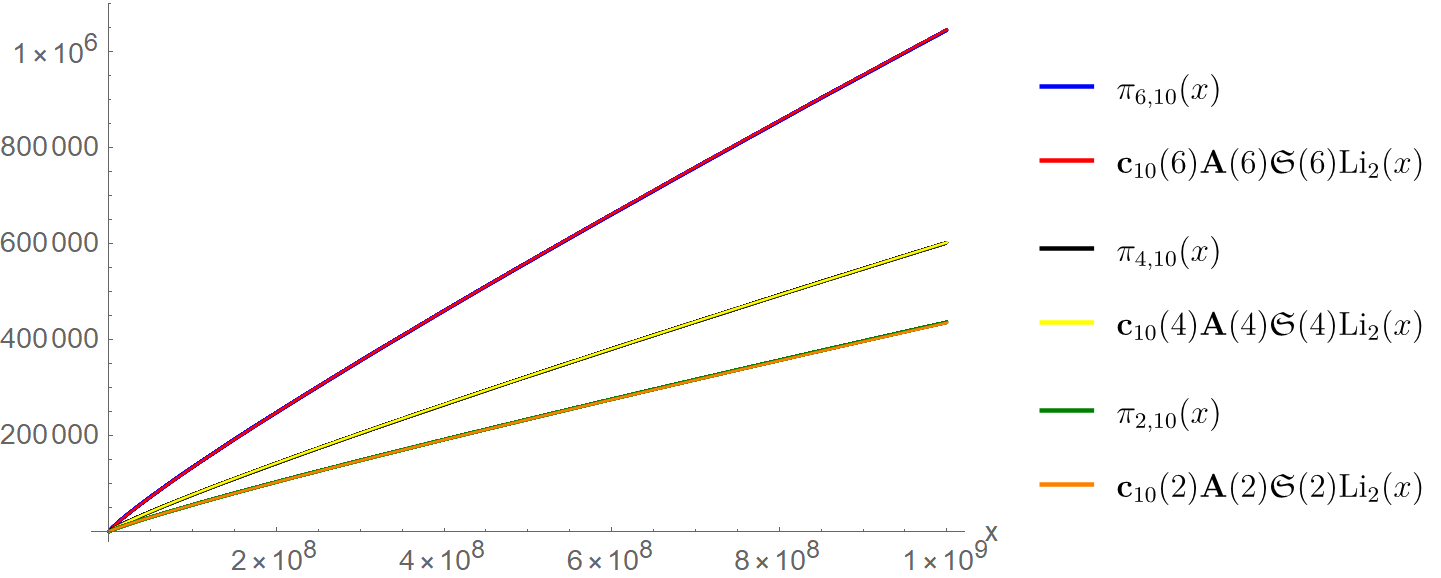}
\caption{\\Plots of $\pi_{2, 10}(x)$, $\pi_{4, 10}(x)$, and $\pi_{6, 10}(x)$ for $ x \leq 10^9 $.}
\end{figure}
\pagebreak
\begin{center}
\captionof{table}{\\
\footnotesize{\textit{Artin Twin Primes} (i.e. $p \in \pi_{2,g}$) up to $x = 10^{9}$ ($2 \leq g \leq 19$)}}
\renewcommand{\arraystretch}{1.2}
\begin{tabular}{|r|r|r|r|r|r|}
\hline
$ d $&$ g $&$ \pi_{d,g}(10^9) $&$ \pi_{d}(10^9) $&$ \pi_{d,g}/\pi_d$&$c_{g}(d)\textbf{A}(d)$\\
\hline
2&2&434292&3424506&0.126818&0.126943\\
2&3&869877&3424506&0.254015&0.253886\\
2&6&433653&3424506&0.126632&0.126943\\
2&7&501029&3424506&0.146306&0.146176\\
2&10&436385&3424506&0.127430&0.126943\\
2&11&485267&3424506&0.141704&0.141338\\
2&12&870116&3424506&0.254085&0.253886\\
2&13&320350&3424506&0.093546&0.093631\\
2&14&434533&3424506&0.126889&0.126943\\
2&15&267775&3424506&0.078193&0.078118\\
2&17&467112&3424506&0.136402&0.136476\\
2&18&435362&3424506&0.127131&0.126943\\
2&19&457300&3424506&0.133537&0.133665\\
\hline
\end{tabular}
\end{center}

\begin{center}
\captionof{table}{\\
\footnotesize{\textit{Artin Cousin Primes} (i.e. $p \in \pi_{4,g}$) up to $x = 10^{9}$ ($5 \leq g \leq 20$)}}
\renewcommand{\arraystretch}{1.2}
\begin{tabular}{|r|r|r|r|r|r|}
\hline
$ d $&$ g $&$ \pi_{d,g}(10^9) $&$ \pi_{d}(10^9) $&$ \pi_{d,g}/\pi_d$&$c_{g}(d)\textbf{A}(d)$\\
\hline
4&5&668522&3424680&0.195207&0.195296\\
4&6&936167&3424680&0.273358&0.273415\\
4&7&369095&3424680&0.107775&0.107709\\
4&10&601828&3424680&0.175732&0.175767\\
4&11&415533&3424680&0.121334&0.121204\\
4&13&518948&3424680&0.151531&0.151251\\
4&14&566971&3424680&0.165554&0.165706\\
4&15&602032&3424680&0.175792&0.175767\\
4&17&504006&3424680&0.147168&0.146974\\
4&19&438500&3424680&0.128041&0.127764\\
4&20&668812&3424680&0.195291&0.195296\\
\hline
\end{tabular}
\end{center}

\pagebreak
\begin{center}
\captionof{table}{\\
\footnotesize{\textit{Artin Sexy Primes} (i.e. $p \in \pi_{6,g}$) up to $x = 10^{9}$ ($2 \leq g \leq 20$)}}
\renewcommand{\arraystretch}{1.2}
\begin{tabular}{|r|r|r|r|r|r|}
\hline
$ d $&$ g $&$ \pi_{d,g}(10^9) $&$ \pi_{d}(10^9) $&$ \pi_{d,g}/\pi_d$&$c_{g}(d)\textbf{A}(d)$\\
\hline
6&2&1045109&6849047&0.152591&0.152588\\
6&5&1492420&6849047&0.217901&0.217982\\
6&6&1045093&6849047&0.152589&0.152588\\
6&7&1290264&6849047&0.188385&0.188491\\
6&10&1044298&6849047&0.152473&0.152588\\
6&11&1142428&6849047&0.166801&0.166745\\
6&13&770054&6849047&0.112432&0.112547\\
6&14&1044000&6849047&0.152429&0.152588\\
6&15&1344686&6849047&0.196331&0.196184\\
6&17&842745&6849047&0.123045&0.123035\\
6&18&1045470&6849047&0.152644&0.152588\\
6&19&1107167&6849047&0.161652&0.161619\\
6&20&1492389&6849047&0.217897&0.217982\\
\hline
\end{tabular}
\end{center}
The remainder of this paper is structured as follows. In Section 2, we write
\begin{align}
\pi_{d,g}(x) &= W_{d,g}(x)+E_{d,g}(x),
\end{align}
where
\begin{equation}\label{model reciprocity}
W_{d,g}(x):= \sum\limits_{\substack{p \leq x\\ p\in \pi_{d}}} w_{g}(p)w_{g}(p+d)
\end{equation}
is viewed as a weighted summation over principal and quadratic Dirichlet characters, and $E_{d,g}(x)$ (defined in (\ref{E sum})) is a weighted summation over higher-order characters.  Similarly flavored identities have previously been employed to study Artin's conjecture over function fields \cite{JensenMurty, Pappalardi}, and serve as the theoretical underpinnings for various Poisson binomial distribution models referred to above \cite{Mo1, Moree3, Moree4, Moree5}.  Guided by these case studies, we expect the contribution from $E_{d,g}(x)$ to be asymptotically negligible (Conjecture \ref{E Conjecture}).  In particular, this suggests that the asymptotic behavior of $\pi_{d,g}(x)$ may be modeled by $W_{d,g}(x)$, in the limit as $x \rightarrow \infty$.\\
\\
In Sections 3 we then employ a conjecture concerning twin primes in arithmetic progressions (Conjecture \ref{conjBHstrong}) to evaluate the sum
\begin{equation}\label{S sum}
S_{d}(x;a,f):=\sum_{\substack{p \leq x \\ p \in \pi_{d} \\ p \equiv a(f)}}\frac{\varphi(p-1)}{p-1}\frac{\varphi(p+d-1)}{p+d-1}.
\end{equation}
This enables us to directly compute $W_{d,2}(x)$ (Corollary \ref{g=2}), yielding Conjecture \ref{main conjecture} for the specific the case $|g|=2$.  The cases in which $|g| \neq 2$ are then dealt with in Section 4, and Section 5 explicitly computes those instances in which the correction factor $\textbf{c}_{g}(d)$ is equal to zero.\\
\\
\textbf{Acknowledgments}: We thank Pieter Moree and Mengzhen Liu for helpful discussions.  Tinkov\'a was supported by the Czech Science Foundation (GA\v{C}R) grant 17-04703Y, as well as the following Charles University projects: UNCE/SCI/022, PRIMUS/20/SCI/002, and SVV-2020-260589.  Waxman was supported by the Czech Science Foundation (GA\v{C}R) grant 17-04703Y, by a visiting postdoctoral fellowship at the Max Planck Institute for Mathematics, and by a Minerva Post-Doctoral Fellowship at the Technische Universit\"at Dresden.  Zindulka was supported by Charles University projects GA UK 742120, PRIMUS/20/SCI/002, and SVV-2020-260589, the Czech Science Foundation GACR grant 21-00420M, and the Charles University Research Centre program UNCE/SCI/022.
\section{Modeling $\pi_{d,g}(x)$ by $W_{d,g}(x)$}
Let $(\Z/p\Z)^{\times}$ denote the finite cyclic group of order $p-1$.  For any $g \in (\Z/p\Z)^{\times}$, we define
\[f_{p}(g) := \left\{
\begin{array}{l l}
1 & \text{ if } ord_{p}(g)=p-1 \\
0 & \text{ otherwise}.
\end{array} \right.\]
From \cite{JensenMurty}, it follows that $f_{p}(g)$ may be written as a \textit{sifting function}
\begin{align*}\label{fourier identity}
f_{p}(g) &= \sum_{k|p-1}\mathcal{S}_{k,p}(g),
\end{align*}
where
\[\mathcal{S}_{k,p}(g):= \frac{\varphi(p-1)}{p-1}\frac{\mu(k)}{\varphi(k)}\sum_{\textnormal{ord}\chi = k}\chi(g).\]
Here $\mu$ is the M{\"o}bius function, and the inner sum runs over the multiplicative characters of $(\Z/p\Z)^{\times}$ with order precisely $k$.  In particular, the \textit{linear} contribution to the above sifting function is given by
\begin{align*}
\mathcal{S}_{1,p}(g)&=\frac{\varphi(p-1)}{p-1}\chi_{0}(g),
\end{align*}
and the \textit{quadratic} contribution to the sifting function is equal to

\begin{align*}
\mathcal{S}_{2,p}(g)&=-\frac{\varphi(p-1)}{p-1}\genleg{g}{p},
\end{align*}
and thus we find that
\begin{align*}
\mathcal{S}_{1,p}(g)+\mathcal{S}_{2,p}(g)= w_{g}(p),
\end{align*}
where $w_{g}(p)$ is as in (\ref{wp}).\\
\\
Fix $g \in \mathcal{H}_{1}$, and write 
\begin{align*}
\pi_{d,g}(x)&= \sum_{p \in \pi_{d}(x)}f_{p}(g)f_{p+d}(g)\\
&= \sum_{k,l \leq x}H_{k,l}(x;d,g),
\end{align*}
where
\begin{align*}
H_{k,l}(x;d,g)&:=\sum_{\substack{p \in \pi_{d}(x)\\ p\equiv 1(k) \\ p+d \equiv 1(l)}}\mathcal{S}_{k}(p)\mathcal{S}_{l}(p+d)\\
&=\sum_{\substack{p \in \pi_{d}(x)\\ p\equiv 1(k) \\ p+d \equiv 1(l)}}\frac{\mu(k)}{\varphi(k)}\frac{\mu(l)}{\varphi(l)}\frac{\varphi(p-1)}{p-1}\frac{\varphi(p+d-1)}{p+d-1}\sum_{\textnormal{ord}\chi = k }\chi(g)\sum_{\textnormal{ord}\rho = l}\rho(g).
\end{align*}
Here $\chi$ is a character on $(\Z/p\Z)^{\times}$ while $\rho$ is a character on $(\Z/(p+d)\Z)^{\times}$.  In particular,
\begin{align*}
H_{1,1}(x;d,g)&=\sum_{\substack{p \in \pi_{d}(x)\\p \nmid g}}\frac{\varphi(p-1)}{p-1}\frac{\varphi(p+d-1)}{p+d-1}\\
H_{2,1}(x;d,g)&=-\sum_{p \in \pi_{d}(x)}\genleg{g}{p}\frac{\varphi(p-1)}{p-1}\frac{\varphi(p+d-1)}{p+d-1}\\
H_{1,2}(x;d,g)&=-\sum_{p \in \pi_{d}(x)}\genleg{g}{p+d}\frac{\varphi(p-1)}{p-1}\frac{\varphi(p+d-1)}{p+d-1}
\end{align*}
and
\[H_{2,2}(x;d,g)=\sum_{p \in \pi_{d}(x)}\genleg{g}{p}\genleg{g}{p+d}\frac{\varphi(p-1)}{p-1}\frac{\varphi(p+d-1)}{p+d-1}.\]
The $H_{1,1}(x;d,g)$ term corresponds to the `linear' contribution, i.e. to a naive model constructed without taking into consideration any quadratic obstructions.  In particular, it is independent of $g$.  The necessary correction factor is then obtained upon including the contribution from quadratic characters, namely by considering the sum
\begin{equation}\label{linear plus quadratic}
H_{1,1}(x;d,g)+H_{1,2}(x;d,g)+H_{2,1}(x;d,g)+H_{2,2}(x;d,g)= W_{d,g}(x).
\end{equation}
It follows that \begin{equation*}
\pi_{d,g}(x)=W_{d,g}(x)+E_{d,g}(x),
\end{equation*}
where
\begin{equation} \label{E sum}
E_{d,g}(x):=\sum_{\substack{k,l \in \mathbb{N}\\(k,l)\neq (1,1)\\(k,l)\neq (2,1)\\(k,l)\neq (1,2)\\(k,l)\neq (2,2)}}H_{k,l}(x;d,g).
\end{equation}
Drawing guidance from similarly flavored identities previously studied \cite{Ciolan, JensenMurty, Pappalardi, Mo1, Moree3, Moree4, Moree5}, we expect $W_{d,g}(x)$ to contribute the main term, and for the contribution coming from $E_{d,g}(x)$ to be of lower order: 
\begin{conjecture}$($\textbf{Moree's Model}$)$\label{E Conjecture}
Fix $d \in 2\N$ and $g \in \mathcal{H}_{1}$.  The contribution of higher-order characters to $\pi_{d,g}(x)$ is asymptotically negligible, i.e.
\[E_{d,g}(x) = o\left(\frac{x}{(\log x)^{2}}\right).\]
\end{conjecture}
Conjecture \ref{E Conjecture} amounts to the heuristic assumption that the distribution of Artin prime pairs amongst the ordinary primes pairs is largely governed by a Poisson binomial distribution.  In practice, we expect $E_{d,g}(x) = O(x^{1-\delta})$ for some $0< \delta < 1$. 
\section{Computing $S_{d}(x;a,f)$}
Let $g \in \N$, and write $g = q_1^{n_{1}}q_2^{n_{2}}\dots q_r^{n_{r}} \in \mathbb{N}$, where the $q_{i}$'s are pairwise distinct prime numbers.  Fix $ d \in 2\N$, and consider the set
\[\varphi_{d}(g):=\{1 \leq \alpha \leq g: (\alpha,g) = (\alpha+d,g)=1)\}.\]
Note that for prime $q$,
\begin{equation}\label{varphi sum on a prime}
|\varphi_{d}(q^{n})| = \left\{
\begin{array}{l l}
q^{n-1}(q-2) & \text{ if } q\nmid d, \\
q^{n-1}(q-1) & \text{ if } q | d.
\end{array} \right.
\end{equation}
Since $|\varphi_{d}|$ is multiplicative, it follows that
\begin{equation}\label{varphi sum}
|\varphi_{d}(g)| = g\prod_{i=1}^{s}\frac{(q_{i}-1)}{q_{i}}\prod_{i=s+1}^{r}\frac{(q_{i}-2)}{q_{i}},
\end{equation}
where $q_{1}\dots, q_{s}|(g,d)$ and $q_{s+1}\dots, q_{r}\not |(g,d)$.\\
\\
Let
\[ \pi_{d}(x; a, f):= \#\{p \in \pi_{d}: p \leq x, \hspace{2mm} p \equiv a\Mod{f}\}.\]
The following generalization of the twin prime conjecture (alternatively: a special case of the \textit{Bateman Horn conjecture} \cite{Aleth, BH}) provides an estimate for the size of $\pi_{d}(x; a, f)$.  We require some uniformity in the error term \cite{Baier}, analogous to how the Siegal Walfisz theorem is employed to count Artin primes in previous variants of Moree's model \cite{Mo1, Moree3, Moree4, Moree5}:
\begin{conjecture} \label{conjBHstrong}
$($\textbf{Prime Gaps in Arithmetic Progressions}$)$ Fix $ B > 0$, and $a,f \in \N$. Then for all $f \leq (\log x)^B$ with $ (a,f) = (a+d,f) = 1 $,
\[\pi_d(x; a, f) = \frac{\mathfrak{S}(d)}{|\varphi_d(f)|}\cdot\textnormal{Li}_{2}(x)+O_{B}\left(\frac{x}{(\log x)^C}\right)\]
for any $C > 2$, where $\mathfrak{S}(d)$ is as in \textnormal{(\ref{twin prime constant})} $($and where the implied constant depends on the choice of $B)$.
\end{conjecture}
\begin{lemma}\label{s computation}
Fix $a,f \in \mathbb{N}$ such that $(a,f)=(a+d,f) = 1$, and moreover assume Conjecture $\textnormal{\ref{conjBHstrong}}$.  Then for any $B>2$,
\begin{align*}
S_{d}(x;a, f) &= \frac{\mathfrak{S}(d)}{|\varphi_{d}(f)|}K_{d,f}(a)\textbf{A}_{f}(d)\textnormal{Li}_{2}(x)+O\left(\frac{x}{(\log x)^{B}}\right),
\end{align*}
where
\begin{align} \label{K constant}
	K_{d,f}(a)& : = \prod_{\substack{p \mid (f,a-1)}}\left(1-\frac{1}{p}\right)	
	\prod_{\substack{p \mid (f,a+d-1)}}
	\left(1-\frac{1}{p}\right),
\end{align}
and
\[\textbf{A}_{f}(d):=\sum_{\substack{m_{1}^{1},m_{2}^{2} \\(m_{1}^{1}m_{2}^{2},f) =1 \\ (m_{1}^{1},m_{2}^{2}) =1\\(m_{1}^{1},1+d)=1\\(m_{2}^{2},1-d)=1}}\frac{\mu(m_{1}^{1})}{m_{1}^{1}\left|\varphi_{d}(m_{1}^{1})\right|}\frac{\varphi((m_{1}^{1},d))}{(m_{1}^{1},d)}\frac{\mu(m_{2}^{2})}{m_{2}^{2}|\varphi_{d}(m_{2}^{2})|}.\]
\end{lemma}
\begin{proof}
By M{\"o}bius inversion,
\begin{equation}\label{mobius inversion}
\frac{\varphi(p-1)}{p-1} = \sum_{k | p-1}\frac{\mu(k)}{k},
\end{equation}
and therefore
\begin{align}\label{inner sum}
S_{d}(x;a, f) &= \sum_{\substack{p \leq x\\
p \in \pi_{d}\\
p \equiv a(f)}}\left(\sum_{M_{1} | p-1}\frac{\mu(M_{1})}{M_{1}}\right)\left(\sum_{M_{2} | p+d-1}\frac{\mu(M_{2})}{M_{2}}\right)\nonumber \\
&= \sum\limits_{\substack{M_{1}, M_{2}\\M_{1},M_{2} \hspace{1mm} \square\textnormal{-free}}}\frac{\mu(M_{1})\mu(M_{2})}{M_{1}M_{2}}\sum\limits_{\substack{p \leq x \\ p \in \pi_{d} \\ p \equiv a(f)\\p \equiv 1(M_{1})\\p \equiv 1-d(M_{2})}}1.
\end{align}

Write
\begin{equation}\label{eqDefkl}
M_{1} = m_{1}^{0} \cdot m_{1}^{1},\hspace{5mm} \textnormal{ and } \hspace{5mm} M_{2} = m_{2}^{0}\cdot m_{2}^{1}\cdot m_{2}^{2},
\end{equation}
where 
\[m_{1}^{0} := (f, M_{1}),\hspace{5mm}  m_{1}^{1} := \frac{M_{1}}{(f,M_{1})},\]
and
\[m_{2}^{0} := (f,M_{2}), \hspace{5mm} m_{2}^{1} := \left(\frac{M_{1}}{(f,M_1)},M_{2}\right) \hspace{2mm}  m_2^{2} := \frac{M_{2}}{(f,M_{2}) \cdot \left(\frac{M_{1}}{(f,M_{1})},M_{2}\right)}.\]
In other words, we group the primes $p$ dividing $M_{2}$ as follows: if $p \mid f$ then $p|m_{2}^{0}$; if $p\nmid f$ but $p| M_{1}$ then $p|m_{2}^{1}$; and if $p\nmid f \cdot M_{1}$ then $p|m_{2}^{2}$.\\
\\
By the Chinese Remainder Theorem (CRT), the congruence conditions

\begin{align}
	p & \equiv a\Mod{f}\nonumber \\
	p & \equiv 1\Mod{M_{1}} \label{p condition} \\
	p & \equiv 1-d\Mod{M_{2}} \nonumber
\end{align}
are simultaneously solvable if and only if
\begin{align}
	a&\equiv 1 \Mod{m_{1}^{0}} \label{condition 1}\\
	a&\equiv 1-d \Mod{m_{2}^{0}} \label{condition 2}\\
	1&\equiv 1-d \Mod{m_{2}^{1}}, \label{condition 3}
\end{align}
where we further employ the fact that
\[1\equiv 1-d \Mod{m_{2}^{1}} \Leftrightarrow 1\equiv 1-d \Mod{(M_{1},M_{2})}\]
whenever (\ref{condition 1}) and (\ref{condition 2}) are both satisfied.  Such a solution is moreover unique modulo
\begin{align*}
LCM(f,M_{1},M_{2}) = f m_{1}^{1} m_{2}^{2},
\end{align*}
and the explicit solution is given by
\begin{align}
	p&\equiv a \Mod{f} \label{condition 4}\\
	p&\equiv 1 \Mod{m_{1}^{1}} \label{condition 5}\\
	p&\equiv 1-d \Mod{m_{2}^{2}}, \label{condition 6}
\end{align}
where $\{f,m_{1}^{1},m_{2}^{2}\}$ now form a pairwise coprime set.\\
\\
It follows that if $A_{M}$ denotes the unique residue class modulo $f m_{1}^{1} m_{2}^{2}$ for which the conditions in (\ref{condition 4})$-$(\ref{condition 6}) are satisfied, then
\begin{align}\label{S sum}
S_{d}(x;a, f) &=\sum_{\substack{m_{1}^{1},m_{2}^{2}\\(m_{1}^{1}m_{2}^{2},f) =1 \\ (m_{1}^{1},m_{2}^{2}) =1}}\frac{\mu(m_{1}^{1})\mu(m_{2}^{2})}{m_{1}^{1}m_{2}^{2}}\sum_{\substack{m_{1}^{0},m_{2}^{0} \\m_{1}^{0} \mid (f, a-1)\\m_{2}^{0} \mid (f, a+d-1)}}\frac{\mu(m_{1}^{0})\mu(m_{2}^{0})}{m_{1}^{0}m_{2}^{0}}\sum_{m_{2}^{1}|(d,m_{1}^{1})}\frac{\mu(m_{2}^{1})}{m_{2}^{1}}\nonumber \\
&\phantom{=}\times \pi_d(x;A_{M}, fm_{1}^{1}m_{2}^{2}).
\end{align}
Next, suppose $A_{M} \in \varphi_{d}\left(fm_{1}^{1}m_{2}^{2}\right)$, i.e.,

\[\left(A_{M},fm_{1}^{1}m_{2}^{2}\right)=\left(A_{M}+d,fm_{1}^{1}m_{2}^{2}\right)=1.\]
By the multiplicative property of the greatest common divisor function,
\begin{align*}
\left(A_{M},fm_{1}^{1}m_{2}^{2}\right)=1 &\Leftrightarrow \left(A_{M},f\right)=\left(A_{M},m_{1}^{1}\right)=(A_{M},m_{2}^{2})=1\\
&\Leftrightarrow (1-d,m_{2}^{2})=1,
\end{align*}
where in the last step we note that $(A_{M},m_{1}^{1}) = 1$ already follows from (\ref{condition 5}), and that $(A_{M},f)=1$ already follows from the assumption of Lemma \ref{s computation}.  Similarly, we find that
\begin{align*}
\left(A_{M}+d,fm_{1}^{1}m_{2}^{2}\right)=1 &\Leftrightarrow (1+d,m_{1}^{1})=1.
\end{align*}
In other words, given the assumptions of Lemma \ref{s computation}, $A_{M} \in \varphi_{d}(f m_{1}^{1} m_{2}^{2})$ implies the additional constraints
\begin{equation}
\label{eqNecCondition}
	(m_{2}^{2},1-d) = 1 \hspace{5mm} \textnormal{ and } \hspace{5mm} (m_{1}^{1},1+d)= 1.
\end{equation}
Next, note that if $A_{M} \not \in \varphi_{d}(f m_{1}^{1} m_{2}^{2})$, then 
\[\pi_d (x;A_{M}, f m_{1}^{1} m_{2}^{2} )\leq 1\]
uniformly in $m_{1}^{1}$ and $m_{2}^{2}$.  Moreover, since the inner sums in (\ref{S sum}) are finite, and since $\pi_{d}(x;A_{M},fm_{1}^{1}m_{2}^{2}) = 0$ whenever either $m_{1}^{1}, m_{2}^{2} > x$, we may bound the contribution from terms for which $A_{M} \not \in \varphi_{d}(fm_{1}^{1}m_{2}^{2})$ by
\begin{equation*}
\ll \sum\limits_{\substack{m_{1}^{1},m_{2}^{2} < x}}\frac{\mu(m_{1}^{1})\mu(m_{2}^{2})}{m_{1}^{1}m_{2}^{2}}\pi_{d}\left(x;A_{M}, fm_{1}^{1}m_{2}^{2}\right)\ll \sum\limits_{m_{1}^{1} < x}\frac{1}{m_{1}^{1}}\sum_{m_{2}^{2} < x}\frac{1}{m_{2}^{2}}\ll (\log x)^{2}.
\end{equation*}
It then follows that
\begin{align}\label{eqSum}
S_{d}(x;a, f) =&\sum_{\substack{m_{1}^{1},m_{2}^{2}\\(m_{1}^{1}m_{2}^{2},f) =1 \\ (m_{1}^{1},m_{2}^{2}) =1\\(m_{1}^{1},1+d)=1\\(m_{2}^{2},1-d)=1}}\frac{\mu(m_{1}^{1})\mu(m_{2}^{2})}{m_{1}^{1}m_{2}^{2}}\sum_{\substack{m_{1}^{0},m_{2}^{0} \\m_{1}^{0} \mid (f, a-1)\\m_{2}^{0} \mid (f, a+d-1)}}\frac{\mu(m_{1}^{0})\mu(m_{2}^{0})}{m_{1}^{0}m_{2}^{0}}\sum_{m_{2}^{1}|(d,m_{1}^{1})}\frac{\mu(m_{2}^{1})}{m_{2}^{1}}\\
&\phantom{=}\times \pi_d(x;A_{M}, fm_{1}^{1}m_{2}^{2})+O\bigg((\log x)^{2}\bigg).\nonumber
\end{align}

Next, suppose $A_{M} \in \varphi_{d}(f m_{1}^{1} m_{2}^{2})$.  Since $\pi_d(x;A_{M},f m_{1}^{1} m_{2}^{2}) \leq \frac{x}{m_{1}^{1}m_{2}^{2}}$, we find that for any fixed $B > 2$,
\begin{align*}
\sum_{\substack{m_{1}^{1}, m_{2}^{2} \\m_{1}^{1} > (\log x)^B}}&\frac{\mu(m_{1}^{1})\mu(m_{2}^{2})}{m_{1}^{1}m_{2}^{2}}\pi_d(x;A_{M}, f m_{1}^{1} m_{2}^{2})&\leq \sum\limits_{m_{1}^{1} > (\log x)^B}\frac{x}{(m_{1}^{1})^2}\sum_{m_{2}^{2} = 1}^\infty\frac{1}{(m_{2}^{2})^2}\\
&= O\left(\frac{x}{(\log x)^B}\right),
\end{align*}
and similarly for the sum over $m_2^{2} > (\log x)^B$.  Upon applying Conjecture \ref{conjBHstrong}, we conclude that
\begin{align}\label{minus error}
\begin{split}
S_{d}(x;a, f) =&\frac{\mathfrak{S}(d)}{|\varphi_{d}(f)|}\cdot \Li_2(x)\times \\
&\sum_{\substack{m_{1}^{1},m_{2}^{2} < (\log x)^B \\(m_{1}^{1}m_{2}^{2},f) =1 \\ (m_{1}^{1},m_{2}^{2}) =1\\(m_{1}^{1},1+d)=1\\(m_{2}^{2},1-d)=1}}\sum_{\substack{m_{1}^{0},m_{2}^{0} \\m_{1}^{0} \mid (f, a-1)\\m_{2}^{0} \mid (f, a+d-1)}}\sum_{m_{2}^{1}|(d,m_{1}^{1})}\frac{\mu(m_{1}^{0})\mu(m_{1}^{1})\mu(m_{2}^{0})\mu(m_{2}^{1})\mu(m_{2}^{2})}{m_{1}^{0}m_{1}^{1}m_{2}^{0}m_{2}^{1}m_{2}^{2}\left|\varphi_{d}(m_{1}^{1} m_{2}^{2})\right|}\\
&+ O\left(\frac{x}{(\log x)^B}\right).
\end{split}
\end{align}
Here we note that all the error terms coming from each application of Conjecture \ref{conjBHstrong} may together be incorporated into $O\left(x(\log x)^{-B}\right)$, since upon fixing $\epsilon > 0$ and setting $C = B+\epsilon$ in each application of Conjecture \ref{conjBHstrong}, we bound the error by
\[
\ll \sum\limits_{\substack{m_1^{1} < (\log x)^B\\m_2^{2} < (\log x)^B}}\frac{\mu(m_1^{1})\mu(m_{2}^{2})}{m_{1}^{1}m_{2}^{2}}\cdot \frac{x}{(\log x)^{C}} \ll (\log \log x)^2\cdot \frac{x}{(\log x)^{B+\epsilon}} \ll \frac{x}{(\log x)^{B}}.\]
Finally, by Mertens' third theorem,
\[\lim_{n \rightarrow \infty}\log n\prod_{p \leq n}\left(1-\frac{1}{p}\right) = e^{-\gamma},\]
where $\gamma$ is the Euler-Mascheroni constant.  It thus follows from (\ref{varphi sum}) that for square-free $n \in \mathbb{N}$,
\[
\frac{1}{|\varphi_d(n)|} \leq \frac{2}{n}\prod_{\substack{p \leq n\\p \neq 2}}\frac{p}{p-2}=\frac{2}{n}\prod_{\substack{p \leq n\\p \neq 2}}\left(\frac{p}{p-1}\right)^2\left(1-\frac{1}{(p-1)^2}\right)^{-1}\sim  \frac{e^{2\gamma}}{\mathfrak{S}(2)}\frac{(\log n)^2}{n},
\]
in the limit as $n \rightarrow \infty$.  Thus
\begin{align*}
	&\sum_{m_{1}^{1} > (\log x)^B}\sum\limits_{m_{2}^{2} = 1}^\infty \frac{\mu(m_{1}^{1})\mu(m_2^{2})}{m_1^1 m_2^2|\varphi_d(m_1^1 m_2^2)|} \ll \sum\limits_{m_1^1 > (\log x)^B}\sum\limits_{m_2^2 = 1}^\infty \frac{\log(m_1^1 m_2^2)^{2}}{(m_1^1m_2^2)^2} \ll (\log x)^{-(B-\epsilon')},
\end{align*}
for any $0 < \epsilon' < 2$, i.e.
\begin{align}
\begin{split}\label{added error}
	\sum\limits_{m_1^1 > (\log x)^B}\sum\limits_{m_2^2 = 1}^\infty \frac{\Li_2(x)}{m_1^1 m_2^2|\varphi_d(m_1^1 m_2^2)|}& = \sum_{m_2^2 > (\log x)^B}\sum\limits_{m_1^1 = 1}^\infty \frac{\Li_2(x)}{m_1^1 m_2^2|\varphi_d(m_1^1 m_2^2)|}\\
&= O\left(\frac{x}{(\log x)^{2+B-\epsilon'}}\right) = O\left(\frac{x}{(\log x)^{B+\epsilon}}\right).
\end{split}
\end{align}
Lemma \ref{s computation} then follows from (\ref{minus error}) and (\ref{added error}), where by (\ref{mobius inversion}) we note that
\[\sum_{m_{2}^{1}|(d,m_{1}^{1})}\frac{\mu(m_{2}^{1})}{m_{2}^{1}} = \frac{\varphi((m_{1}^{1},d))}{(m_{1}^{1},d)},\]
and that
\begin{align*}
\sum_{\substack{m_{1}^{0},m_{2}^{0}\\ m_{1}^{0}|(f,a-1)\\m_{2}^{0}|(f,a+d-1)}}\frac{\mu(m_{1}^{0})\mu(m_{2}^{0})}{m_{1}^{0}m_{2}^{0}} &= \frac{\varphi((f,a-1))}{(f,a-1)}\frac{\varphi((f,a+d-1))}{(f,a+d-1)}\\
&=\prod_{\substack{p \mid (f,a-1)}}\left(1-\frac{1}{p}\right)	
	\prod_{\substack{p \mid (f,a+d-1)}}
	\left(1-\frac{1}{p}\right).
\end{align*}

\end{proof}
\begin{remark} \textnormal{In fact, a slightly weaker form of Conjecture \ref{conjBHstrong} is sufficient to derive Conjecture \ref{main conjecture}.  Specifically, to compute the asymptotic growth of $S_{d}(x;a,f)$ in Lemma \ref{s computation} (and subsequently to compute the asymptotic growth of $W_{d,g}(x)$ in Theorem \ref{A constant}), one need only assume Conjecture \ref{conjBHstrong} for some fixed $B, C > 2$.}
\end{remark}
\subsection{A Product Formula for $S_{d}(x;a,f)$}
Let
\begin{equation}\label{main constant}
\textbf{A}(d)= \prod_{p \textnormal { prime}}\rho_{d}(p),
\end{equation}
where
\begin{equation*}
\rho_{d}(p) := \left\{
\begin{array}{l l l l}
1+\frac{1}{p^2(p-1)}-\frac{2}{p(p-1)} & \text{ if } p\mid d\\
1-\frac{1}{p(p-2)} & \text{ if } p\mid d\pm 1\\
1-\frac{2}{p(p-2)} & \text{ if } p\nmid d(d^2 - 1).
\end{array} \right.
\end{equation*}

\begin{cor}\label{S constant}
Fix $a,f \in \mathbb{N}$ such that $(a,f)=(a+d,f) = 1$, and moreover assume Conjecture \ref{conjBHstrong}.  Then for any $B>2$,
\begin{equation*}
S_{d}(x;a, f) =\mathfrak{S}(d)\textbf{e}_{f}(d)\textbf{A}(d) K_{d,f}(a)\textnormal{Li}_{2}(x)+O\left(\frac{x}{(\log x)^{B}}\right),
\end{equation*}
where
\begin{equation}
\textbf{e}_{f}(d):= \frac{1}{|\varphi_{d}(f)|}\prod_{\substack{p|f}}\rho_{d}(p)^{-1}.
\end{equation}
In particular,
\[S_{d}(x;1,1) = H_{1,1}(x;d,g)= \textbf{A}(d) \mathfrak{S}(d)\textnormal{Li}_{2}(x)+O\left(\frac{x}{(\log x)^{B}}\right).\]
\end{cor}

\begin{proof}
Let $ \eta$ be a multiplicative function, and note that 
\[
	\sum_{\substack{n = 1\\n\text{ square-free}}}^\infty \eta(n) = \prod_p \left(1+\eta(p)\right),
\]
whenever the series $\sum_{n=1}^{\infty} \eta(n)$ is absolutely convergent (see \cite[Theorem~11.6]{Ap}).  
Upon defining the multiplicative functions
\begin{align*}
	\eta_1(n)& := \begin{cases}
		\frac{\mu(n)}{n\left|\varphi_d(n)\right|}\frac{\varphi((n,d))}{(n,d)}&\text{ if } (n,f)=1 \text{ and }(n, 1+d)=1\\
		0 &\text{ otherwise},
	\end{cases}\\
	\eta_2(n)& := \begin{cases}
		\frac{\mu(n)}{n|\varphi_{d}(n)|}&\text{ if }(n,f)=1 \text{ and }(n, 1-d)=1\\
		0&\text{ otherwise},
	\end{cases}
\end{align*}
it then follows that 
\begin{align*}
\textbf{A}_{f}(d) &= \sum_{\substack{n, m \\n,m \text{ square-free}\\(n,m) = 1}}\eta_1(n) \eta_2(m)= \prod_{p} \left(1+\eta_1(p)+\eta_2(p)\right).
\end{align*}
Upon noting that for any $p \nmid f$,
\begin{equation}
	\eta_{1}(p) = \left\{
\begin{array}{l l}
0 & \text{ if } p|d+1 \\
-\frac{1}{p^{2}} & \text{ if } p|d \\
-\frac{1}{p(p-2)} & \text{ if } p\nmid d(d+1),
\end{array} \right.
\end{equation}
and
\begin{equation}
	\eta_{2}(p) = \left\{
\begin{array}{l l}
0 & \text{ if } p|d-1 \\
-\frac{1}{p(p-1)} & \text{ if } p|d \\
-\frac{1}{p(p-2)} & \text{ if } p\nmid d(d-1),
\end{array} \right.
\end{equation}
we conclude that 
\[\textbf{A}_{f}(d) = \prod_{p \nmid f} \rho_{d}(p) = \textbf{A}(d)\textbf{e}_{f}(d),\]
from which Corollary \ref{S constant} now follows. 
\end{proof}

\subsection{The Correction Factor for $g=2$}
We can now compute $W_{d,g}(x)$ in the particular case that $g=2$ (the case $g=-2$ is very similar):
\begin{cor}\label{g=2}
Fix $g=2$ and assume Conjecture \textnormal{\ref{conjBHstrong}}.  Then 
\begin{equation}\label{g=2 equation}
W_{d,2}(x)=\textbf{c}_{2}(d)\textbf{A}(d)\mathfrak{S}(d)\textnormal{Li}_{2}(x)+O\left(\frac{x}{(\log x)^{C}}\right)
\end{equation}
where $\textbf{c}_{2}(d)$ is as in \textnormal{(\ref{correction factor g=2})}.
\end{cor}
\begin{proof}
Since
\begin{equation}
	\genleg{2}{p} =\left\{
\begin{array}{l l}
1 & \text{ if } p \equiv 1,7(8) \\
-1 & \text{ if } p \equiv 3,5(8),
\end{array} \right.
\end{equation}
we find that
\begin{equation}
	W_{d,2}(x) =4 \times \left\{
\begin{array}{l l}
S_{d}(x;3,8)+S_{d}(x;5,8) & \text{ if } d \equiv 0(8) \\
S_{d}(x;3,8) & \text{ if } d \equiv 2(8) \\
0 & \text{ if } d \equiv 4(8) \\
S_{d}(x;5,8) & \text{ if } d \equiv 6(8).
\end{array} \right.
\end{equation}
Equation (\ref{g=2 equation}) now follows from Corollary \ref{S constant} upon noting that 
\begin{align*}
S_{d}(x;3,8)=S_{d}(x;5,8)&=\frac{1}{4}\textbf{A}(d)\mathfrak{S}(d)\textnormal{Li}_{2}(x)+O\left(\frac{x}{(\log x)^{C}}\right).
\end{align*}
\end{proof}
\section{The Correction Factor for $g\neq 2$}
\subsection{Quadratic Reciprocity}To compute $\textbf{c}_{g}(d)$ when $g \neq 2$, we make extensive use of quadratic reciprocity.  For $g \in \mathcal{H}_{1} \setminus \{\pm 2\}$, write $g = g_{0}g_{s}^2$, where $g_{0}$ is square-free, and define
\begin{equation}\label{g tilde}
\tilde{g} := \left\{
\begin{array}{l l}
|g_{0}| & \text{ if } g_{0} \textnormal{ is odd}, \\
\left| \frac{g_{0}}{2}\right| & \text{ if } g_{0} \textnormal{ is even}. \\
\end{array} \right.
\end{equation}

Fix $d \in 2\N$ and $\varepsilon_{1}, \varepsilon_{2} \in \{0, \pm 1\}$.  For odd $n \in \mathbb{N}$, define
\[
	\delta_{\varepsilon_{1},\varepsilon_{2}}(\tilde{g},d)(n) := \begin{cases}
		1 &\text {if }\genleg{n}{\tilde{g}} = \varepsilon_{1} \textnormal{ and }\genleg{n+d}{\tilde{g}} = \varepsilon_{2}\\
		0&\text{ otherwise},
	\end{cases}
\]
where $(n/\tilde{g})$ now denotes the \textit{Jacobi} symbol.

\begin{lemma}\label{Residue Classes}
Fix $p \in \pi_{d}$ and suppose $(p,g_{s})=(p+d,g_{s})=1$. Then

\begin{equation}\label{odd g residue}
\Delta_{d,g}(p) = \left\{
\begin{array}{l l  l}
\delta_{-1,-1}(\tilde{g},d)(p) & \text{ if } g_{0} \equiv 1(4)\\
\delta_{T_{1}(p),T_{1}(p+d)}(\tilde{g},d)(p) & \text{ if } g_{0} \equiv 3(4)\\
\delta_{T_{2}(p),T_{2}(p+d)}(\tilde{g},d)(p) & \text{ if } g_{0} \equiv 2(8)\\
\delta_{T_{3}(p),T_{3}(p+d)}(\tilde{g},d)(p) & \text{ if } g_{0} \equiv 6(8),
\end{array} \right.
\end{equation}

where $T_{1}$, $T_{2}$, and $T_{3}$ are as in \textnormal{(\ref{T1})}, \textnormal{(\ref{T2})}, \textnormal{(\ref{T3})}, respectively.
\end{lemma}

\begin{proof}
Since $(p,g_s)=(p+d,g_s)=1$, we have $\genleg{g}{p} =\genleg{g_0}{p}$ and $\genleg{g}{p+d} =\genleg{g_0}{p+d}$.  The idea of the proof is to use quadratic reciprocity to write $\genleg{g}{p}$ and $\genleg{g}{p+d}$ in terms of $\genleg{p}{\tilde{g}}$ and $\genleg{p+d}{\tilde{g}}$. The result depends heavily on the particular residue classes of $n$, $d$, and $g_{0}$.\\
\\
First, suppose $g_0 \equiv 1 \Mod{4}$.  If $g_0 > 0$, then
\begin{equation}\label{1 mod 4 flip}
	\genleg{g_0}{p} =\genleg{p}{g_0} =\genleg{p}{\tilde{g}},
\end{equation}
and similarly if $g_0 < 0$, then
\[\genleg{g_0}{p} = \genleg{-1}{p}\genleg{\tilde{g}}{p} = \genleg{-1}{p}(-1)^{\frac{p-1}{2}}\genleg{p}{\tilde{g}} = \genleg{p}{\tilde{g}}.\]
It follows that

\begin{equation}\label{odd g residue}
\Delta_{d,g}(p) =  \delta_{-1,-1}(\tilde{g},d)(p).
\end{equation}
Next, suppose $g_0 \equiv 3 \Mod{4}$.  If $g_0 > 0$, then
\begin{equation}\label{3 mod 4 flip}
	\genleg{g_0}{p} = \left\{
\begin{array}{l l}
-\genleg{p}{\tilde{g}} & \text{ if } p \equiv 3(4) \\
\genleg{p}{\tilde{g}} & \text{ if } p \equiv 1(4),
\end{array} \right.
\end{equation}
and similarly if $g_0 < 0$, then
\begin{equation}\label{3 mod 4 flip negative}
	\genleg{g_0}{p}=\genleg{-1}{p}\genleg{\tilde{g}}{p}	=\genleg{-1}{p}\genleg{p}{\tilde{g}} = \left\{
\begin{array}{l l}
-\genleg{p}{\tilde{g}} & \text{ if } p \equiv 3(4) \\
\genleg{p}{\tilde{g}} & \text{ if } p \equiv 1(4).
\end{array} \right.
\end{equation}
Upon defining
\begin{equation}\label{g tilde}
T_{1}(n) := \left\{
\begin{array}{l l}
1 & \text{ if } n \equiv 3(4) \\
-1 & \text{ if } n \equiv 1(4)\\
\end{array} \right.
\end{equation}
for odd $n \in \mathbb{N}$, it follows that for $g_{0}\equiv 3(4)$,

\begin{align*}
\genleg{g_{0}}{p} = -1 & \Leftrightarrow \genleg{p}{\tilde{g}} = T_{1}(p),
\end{align*}
and therefore that
\begin{equation}\label{T1}
\Delta_{d,g}(p) = 
\delta_{T_{1}(p),T_{1}(p+d)}(\tilde{g},d)(p).
\end{equation}
Next, suppose $g_0$ is even.  If $g_0> 0$, then
\begin{equation}\label{even p flip}
	\genleg{g_0}{p} = \genleg{2}{p}\genleg{\tilde{g}}{p} = \left\{
\begin{array}{l l}
\genleg{\tilde{g}}{p} & \text{ if } p \equiv 1,7(8) \\
- \genleg{\tilde{g}}{p} & \text{ if } p \equiv 3,5(8),
\end{array} \right.
\end{equation}
while if $g_0 < 0$,
\begin{equation}\label{even p flip negative}
	\genleg{g_0}{p} = \genleg{2}{p}\genleg{-1}{p}\genleg{\tilde{g}}{p} = \left\{
\begin{array}{l l}
\genleg{\tilde{g}}{p} & \text{ if } p \equiv 1,3(8) \\
- \genleg{\tilde{g}}{p} & \text{ if } p \equiv 5,7(8).
\end{array} \right.
\end{equation}
It follows that when  $g_{0}\equiv 2(8)$,
\begin{align*}
\genleg{g_{0}}{p} = -1 & \Leftrightarrow \genleg{p}{\tilde{g}} = T_{2}(p),
\end{align*}
where for odd $n \in \mathbb{N}$,
\begin{equation}\label{T2}
T_{2}(n) := \left\{
\begin{array}{l l}
1 & \text{ if } n \equiv 3,5(8) \\
-1 & \text{ if } n \equiv 1,7(8).\\
\end{array} \right.
\end{equation}
Hence
\begin{equation}
\Delta_{d,g}(p) = \delta_{T_{2}(p),T_{2}(p+d)}(\tilde{g},d)(p).
\end{equation}
Similarly, for $g_{0} \equiv 6(8)$, we find that
\begin{align*}
\genleg{g_{0}}{p} = -1 & \Leftrightarrow \genleg{p}{\tilde{g}} = T_{3}(p),
\end{align*}
where
\begin{equation}\label{T3}
T_{3}(n) := \left\{
\begin{array}{l l}
1 & \text{ if } n \equiv 5,7(8) \\
-1 & \text{ if } n \equiv 1,3(8).\\
\end{array} \right.
\end{equation}
Hence
\begin{equation}\label{odd g residue}
\Delta_{d,g}(p) = \delta_{T_{3}(p),T_{3}(p+d)}(\tilde{g},d)(p),
\end{equation}
as desired.
\end{proof}

\subsection{Properties of $\Omega_{d,f}$}
Fix $ d \in 2\N$, and $f \in \N$ to be odd and square-free.  Define

\begin{equation}\label{omega constant}
\Omega_{d,f} (\varepsilon_{1},\varepsilon_{2}):=\textbf{e}_{f}(d)\sum_{a \in \Phi_{\varepsilon_{1},\varepsilon_{2}}(f, d)}K_{d,f}(a),
\end{equation}
where
\begin{equation}\label{Phi function}
\Phi_{\varepsilon_{1},\varepsilon_{2}}(f,d):= \bigg \{1 \leq  \alpha \leq f: \genleg{\alpha}{f} = \varepsilon_{1} \textnormal{ and } \genleg{\alpha+d}{f} = \varepsilon_{2}\bigg \},
\end{equation}
and $ \varepsilon_1, \varepsilon_2 \in \{-1,1\}.$\\
\\
Furthermore, if $1 \leq a \leq f$ is the unique solution to the congruence conditions

\begin{align}
	a & \equiv a_{1} \Mod{f_{1}}\nonumber \\
	a & \equiv a_{2} \Mod{f_{2}},
\end{align}
where $(f_{1},f_{2})=1$ and $f=f_{1}f_{2}$, we moreover define
\[\Omega_{a_{2}}^{f_{2}}(\varepsilon_{1},\varepsilon_{2}):=\textbf{e}_{f}(d)\sum_{a_{1} \in \Phi_{\varepsilon_{1},\varepsilon_{2}}(f_{1}, d)}K_{d, f}(a).\]
\begin{lemma}\label{omega relation}
\[\Omega_{a_{2}}^{f_{2}}(\varepsilon_{1},\varepsilon_{2}) = e_{f_{2}}(d)K_{d,f_{2}}(a_{2})\Omega_{d,f_{1}}(\epsilon).\]
\end{lemma}

\begin{proof}
Note that

\[e_{f}(d) = e_{f_{1}}(d)e_{f_{2}}(d)\]
and
\[K_{d,f}(a) = K_{d,f_{1}}(a_{1})K_{d,f_{2}}(a_{2}).\]
It follows that

\begin{align*}
\Omega_{a_{2}}^{f_{2}}(\varepsilon_{1},\varepsilon_{2})&= \sum_{\substack{a_{1} \in \Phi_{\epsilon}(f_{1},d)}}e_{f}(d)K_{d,f}(a)\\
&= e_{f_{2}}(d)K_{d,f_{2}}(a_{2})e_{f_{1}}(d)\sum_{a_{1} \in \Phi_{\epsilon}(f_{1},d)}K_{d,f_{1}}(a_{1})\\
&=e_{f_{2}}(d)K_{d,f_{2}}(a_{2})\Omega_{d,f_{1}}(\epsilon),
\end{align*}
as desired.
\end{proof}
The following lemma provides a simpler description of $\Omega_{d,f}(\varepsilon_{1},\varepsilon_{2})$ that is useful for numerical computations:

\begin{lemma} \label{lemma:coeff}
Fix $d \in 2\mathbb{N}$ and $f \in \mathbb{N}$ to be odd and square-free.  Then
\[
\Omega_{d,f}(\varepsilon_{1},\varepsilon_{2}) = \frac{\textbf{e}_{f}(d)}{4}\left(c(f)+\varepsilon_1c_1(f)+\varepsilon_2c_2(f)+\varepsilon_1\varepsilon_2c_{12}(f)\right),
\]
where $c, c_{1},c_{2},c_{12}$ are multiplicative functions defined on primes $q \in \mathbb{N}$ by
\begin{align*}
    c(q)& := \begin{cases}
		q-1-\frac{2}{q}+\frac{1}{q^2} &\text{ if }q \mid d\\        
        q-2-\frac{1}{q}&\text{ if }q \mid d \pm 1\\
        q-2-\frac{2}{q}&\text{ if }q \nmid d(d\pm 1),
    \end{cases}\\
    c_1(q)& := \begin{cases}
    -\frac{2}{q}+\frac{1}{q^2} &\text{ if }q \mid d\\  
        -1-\frac{1}{q}\genleg{1-d}{q}&\text{ if }q \mid d+1\\
        -\genleg{-d}{q}-\frac{1}{q}\left(1+\genleg{1-d}{q}\right)&\text{ if }q \nmid d(d+1),
    \end{cases}\\
    c_2(q)& := \begin{cases}
        -\frac{2}{q}+\frac{1}{q^2} &\text{ if }q \mid d\\ 
        -1-\frac{1}{q}\genleg{1+d}{q}&\text{ if }q \mid d-1\\
        -\genleg{d}{q}-\frac{1}{q}\left(1+\genleg{1+d}{q}\right)&\text{ if }q \nmid d(d-1),\\
    \end{cases}\\
    c_{12}(q)& :=\begin{cases}
        q-1-\frac{2}{q}+\frac{1}{q^{2}} &\text{ if }q \mid d\\ 
        -1-\frac{1}{q}\left(\genleg{1+d}{q}+\genleg{1-d}{q}\right) &\text{ if }q \nmid d.
        \end{cases}
\end{align*}

\end{lemma}

\begin{proof}
Note that

\[|\Phi_{\varepsilon_{1},\varepsilon_{2}}(f,d)| = \frac{1}{4}\sum\limits_{\substack{1 \leq \alpha \leq f\\(\alpha, f) = 1\\(\alpha+d, f) = 1}}\left(1+\varepsilon_1\genleg{\alpha}{f}\right)\left(1+\varepsilon_2\genleg{\alpha+d}{f}\right),\]
and therefore that
\begin{align*}
&\Omega_{d,f}(\varepsilon_{1},\varepsilon_{2}) =  \textbf{e}_{f}(d) \sum\limits_{a\in\Phi_{\varepsilon_{1},\varepsilon_{2}}(f,d)}\prod\limits_{p \mid (a-1, f)}\left(1-\frac{1}{p}\right)\prod\limits_{p \mid (a+d-1, f)}\left(1-\frac{1}{p}\right)\\
    &= \frac{\textbf{e}_{f}(d)}{4}\sum\limits_{\substack{1 \leq a \leq f\\(a, f) = 1\\(a+d, f) = 1}}\left(1+\varepsilon_1\genleg{a}{f}\right)\left(1+\varepsilon_2\genleg{a+d}{f}\right)\sum\limits_{k \mid (a-1,f)}\sum\limits_{l \mid (a+d-1, f)}\frac{\mu(k)}{k}\frac{\mu(l)}{l}\\
&= \frac{\textbf{e}_{f}(d)}{4}\sum\limits_{k, l \mid f}\frac{\mu(k)}{k}\frac{\mu(l)}{l}\sum\limits_{\substack{1 \leq a \leq f\\(a, f) = 1\\(a+d, f) = 1\\a\equiv 1\Mod{k}\\a\equiv 1-d\Mod{l}}}\left(1+\varepsilon_1\genleg{a}{f}+\varepsilon_2\genleg{a+d}{f}+\varepsilon_1\varepsilon_2\genleg{a}{f}\genleg{a+d}{f}\right).
\end{align*}
Split the inner sum into four parts, and then consider the case when $f=q$ is prime.  Employing that fact that
\begin{equation}\label{prime sum}
\sum\limits_{\substack{1 \leq \alpha \leq q\\(\alpha+d, q) = 1}}\genleg{\alpha}{q} = \sum\limits_{1 \leq \alpha \leq q}\genleg{\alpha}{q} - \genleg{-d}{q}\ = \left\{
\begin{array}{l l}
0 & \text{ if } q|d \\
-\genleg{-d}{q} & \text{ if } q \nmid d,
\end{array} \right.
\end{equation}
as well as that
\begin{align}\label{double prime sum}
\begin{split}
\sum\limits_{1 \leq \alpha \leq q}\genleg{\alpha}{q}\genleg{\alpha+d}{q}&=\sum\limits_{1 \leq \alpha < q}\genleg{\alpha^{2}}{q}\genleg{ \alpha^{-1}(\alpha+d)}{q}\\
&= \sum\limits_{1 \leq \alpha < q}\genleg{1+\alpha^{-1}d}{q} = \left\{
\begin{array}{l l}
q-1 & \text{ if } q|d, \\
-1 & \text{ if } q \nmid d,
\end{array} \right.
\end{split}
\end{align}
we find that the sums are equal to $c(q)$, $ c_1(q) $, $ c_2(q) $ and $c_{12}(q) $, respectively.  By CRT, we moreover find that for $f = f_{1}f_{2}$, where $(f_{1},f_{2}) = 1$,
\begin{align*}
\textbf{e}_{f}(d)\sum_{\substack{(a,f)=1}}K_{d,f}(a)=\textbf{e}_{f_{1}}(d)\textbf{e}_{f_{2}}(d)\sum_{\substack{(a_{1},f_{1})=1\\(a_{1}+d,f_{1})=1}}K_{d,f_{1}}(a_{1})\sum_{\substack{(a_{2},f_{2})=1\\(a_{2}+d,f_{2})=1}}K_{d,f_{2}}(a_{2}).
\end{align*}
In other words, upon extending $c(q)$ (and similarly $c_{1},c_{2}$ and $c_{12}$) to a multiplicative function on $\N$, we obtain the desired count.
\end{proof}

\subsection{Computing the Correction Factor}
For $g \in \mathcal{H}_{1}$, define

\begin{equation}\label{correction factor odd}
\textbf{c}_{g}(d) := 2\times \left\{
\begin{array}{l l l l l l l l}
2\cdot \Omega_{d,\tilde{g}}(-1,-1)& \text{ if } g_{0} \equiv 1(4)\\
\Omega_{d,\tilde{g}}(-1,-1)+\Omega_{d,\tilde{g}}(1,1)& \text{ if } g_0 \equiv 3(4) \textnormal{ and } d \equiv 0(4)\\
\Omega_{d,\tilde{g}}(-1,1)+\Omega_{d,\tilde{g}}(1,-1) & \text{ if } g_0 \equiv 3(4) \textnormal{ and } d \equiv 2(4)
\end{array} \right.
\end{equation}
when $g_{0}$ is odd, and
\begin{equation}\label{correction factor even}
\textbf{c}_{g}(d) :=  \left\{
\begin{array}{l l l l}
2 \left(\Omega_{d,\tilde{g}}(-1,-1)+\Omega_{d,\tilde{g}}(1,1)\right) & \text{ if } d \equiv 0(8)\\
2 \left(\Omega_{d,\tilde{g}}(-1,1)+\Omega_{d,\tilde{g}}(1,-1)\right) & \text{ if } d \equiv 4(8)\\
\Omega_{d,\tilde{g}}(-1,1)+\Omega_{d,\tilde{g}}(1,1)+\Omega_{d,\tilde{g}}(1,-1)+\Omega_{d,\tilde{g}}(-1,-1) & \text{ if } d \equiv 2,6(8)
\end{array} \right.
\end{equation}
when $g_{0}$ is even.\\
\\
Conjecture \ref{main conjecture} now follows from Conjecture \ref{E Conjecture}, together with the following Corollary:

\begin{cor}\label{A constant}
Fix $d \in 2\mathbb{N}$ and $g \in \mathcal{H}_{1}$.  If Conjecture \textnormal{\ref{conjBHstrong}} holds, then 
\[W_{d,g}(x) = \mathbf{A}(d)\textbf{c}_{g}(d)\mathfrak{S}(d) \textnormal{Li}_{2}(x)+O\left(\frac{x}{(\log x)^{C}}\right),\]
\end{cor}

\begin{proof}
By Lemma \ref{Residue Classes} and Corollary \ref{S constant}, we find that
\[W_{d,g}(x) = 4 c_{d,g}\cdot \textbf{A}(d)\mathfrak{S}(d) \textnormal{Li}_{2}(x)+O\left(\frac{x}{(\log x)^{C}}\right),\]
where
\begin{equation*}
c_{d,g} := \left\{
\begin{array}{l l l l l l l l}
\Omega_{1}^{2}(-1,-1)& \text{ if } g_{0} \equiv 1(4)\\
\sum_{t= 1,3}\Omega_{t}^{4}(T_{1}(t),T_{1}(t+d))& \text{ if } g_0 \equiv 3(4)\\
\sum_{t= 1,3,5,7}\Omega_{t}^{8}(T_{2}(t),T_{2}(t+d)) & \text{ if } g_{0} \equiv 2(8)\\
\sum_{t= 1,3,5,7}\Omega_{t}^{8}(T_{3}(t),T_{3}(t+d)) & \text{ if } g_{0} \equiv 6(8).
\end{array} \right.
\end{equation*}
The desired result now follows from Lemma \ref{omega relation}.
\end{proof}

\section{Determining Cases for Which $\textbf{c}_{g}(d)=0$}
\subsection{Computing $|\Phi_{\varepsilon_{1},\varepsilon_{2}}(\tilde{g},d)|$}

\begin{lemma}\label{Lemma1}
Fix $\tilde{g} \in \N_{> 1}$ to be odd and square-free, and write $\tilde{g} = q_1 q_{2}\dots q_r$, where the $q_{i}$'s are pairwise distinct prime numbers.  Then 
\[|\Phi_{\varepsilon_{1},\varepsilon_{2}}(\tilde{g},d)| = \frac{1}{4}\times \left\{
\begin{array}{l l}
\phantom{\bigg[}|\varphi_{d}(\tilde{g})|+(-1)^{r}\bigg(\genleg{-d}{\tilde{g}}\varepsilon_1+\genleg{d}{\tilde{g}}\varepsilon_2+\varepsilon_1\varepsilon_2\bigg)\phantom{\bigg[} & \text{ if } (\tilde{g},d) = 1 \\
\phantom{\bigg[} |\varphi_{d}(\tilde{g})|+(-1)^{r-s}\varepsilon_1\varepsilon_2\prod_{i=1}^{s}(q_{i}-1)\phantom{\bigg[} & \text{ if } (\tilde{g},d)\neq 1,
\end{array} \right.\]
where $(\tilde{g},d)= q_{1}\dots q_{s}$.
\end{lemma}
\begin{proof}
Note that
\begin{align}
\nonumber
&4\cdot |\Phi_{\varepsilon_{1},\varepsilon_{2}}(\tilde{g},d)| = \sum\limits_{\substack{1 \leq \alpha \leq \tilde{g}\\(\alpha, \tilde{g}) = 1\\(\alpha+d, \tilde{g}) = 1}}\left(1+\varepsilon_1\genleg{\alpha}{\tilde{g}}\right)\left(1+\varepsilon_2\genleg{\alpha+d}{\tilde{g}}\right)\\
    \label{eq4sums}
    & = |\varphi_{d}(\tilde{g})|+\varepsilon_1\sum\limits_{\substack{1 \leq \alpha \leq \tilde{g}\\(\alpha+d, \tilde{g}) = 1}}\genleg{\alpha}{\tilde{g}}+\varepsilon_2\sum\limits_{\substack{1 \leq \alpha \leq \tilde{g}\\(\alpha, \tilde{g}) = 1}}\genleg{\alpha+d}{\tilde{g}}+\varepsilon_1\varepsilon_2\sum\limits_{1 \leq \alpha \leq \tilde{g}}\genleg{\alpha}{\tilde{g}}\genleg{\alpha+d}{\tilde{g}}.
\end{align}

By (\ref{prime sum}) and multiplicativity, it follows that

\begin{equation}\label{sum 2}
\sum\limits_{\substack{1 \leq \alpha \leq \tilde{g}\\(\alpha+d, \tilde{g}) = 1}}\genleg{\alpha}{\tilde{g}}  = \left\{
\begin{array}{l l}
0 & \text{ if } (\tilde{g},d)\neq 1, \\
(-1)^{r}\genleg{-d}{\tilde{g}} & \text{ if } (\tilde{g}, d)=1,
\end{array} \right.
\end{equation}
and similarly that 
\begin{equation}\label{sum 3}
\sum\limits_{\substack{1 \leq \alpha \leq \tilde{g}\\(\alpha, \tilde{g}) = 1}}\genleg{\alpha+d}{\tilde{g}}  = \left\{
\begin{array}{l l}
0 & \text{ if } (\tilde{g},d)\neq 1, \\
(-1)^{r}\genleg{d}{\tilde{g}} & \text{ if } (\tilde{g}, d)=1.
\end{array} \right.
\end{equation}
By (\ref{double prime sum}) and multiplicativity, we moreover obtain
\begin{equation}\label{sum 4}
\sum\limits_{1 \leq \alpha \leq \tilde{g}}\genleg{\alpha}{\tilde{g}}\genleg{\alpha+d}{\tilde{g}}= (-1)^{r-s}\prod_{i=1}^{s}(q_{i}-1).
\end{equation}
Lemma \ref{Lemma1} now follows upon inserting (\ref{sum 2}), (\ref{sum 3}), and (\ref{sum 4}), into (\ref{eq4sums}).
\end{proof}

We are particularly interested in determining the cases for which $|\Phi_{\varepsilon_{1},\varepsilon_{2}}(\tilde{g},d)| = 0$.
\begin{lemma} \label{zerocases}
Fix $d \in 2\mathbb{N}$ and $\tilde{g} \in \N_{> 1}$ to be odd and square-free. Then
\begin{enumerate}
\item $|\Phi_{-1,-1}(\tilde{g},d)|=0$ if and only if either
\begin{itemize}
\item $\tilde{g}=5$ and $d\equiv2,3\pmod{5}$, or
\item $\tilde{g}\nmid d$ and $\tilde{g}|3d$
\end{itemize}
\item $|\Phi_{1,1}(\tilde{g},d)|=0$ if and only if
\begin{itemize}
\item $\tilde{g}=5$ and $d\equiv 1,4 \pmod 5$, or
\item $\tilde{g}\nmid d$ and $\tilde{g}|3d$,
\end{itemize}
\item $|\Phi_{-1,1}(\tilde{g},d)|=0$ if and only if $\tilde{g}|d$,
\item $|\Phi_{1,-1}(\tilde{g},d)|=0$ if and only if $\tilde{g}|d$.
\end{enumerate}
\end{lemma}

\begin{proof}
As above, write $\tilde{g}=q_1\ldots q_r$, where $q_i$ are pairwise distinct prime numbers, and let $(\tilde{g},d)=q_1\ldots q_s$.\\
\\
1) If $(\tilde{g},d)=1$, then by Lemma \ref{Lemma1},
\[|\Phi_{-1,-1}(\tilde{g},d)|=0 \Leftrightarrow |\varphi_{d}(\tilde{g})|+(-1)^{r}\bigg(1-\genleg{-d}{\tilde{g}}-\genleg{d}{\tilde{g}}\bigg) = 0,\]
which occurs if and only if, either: $i)$: $\tilde{g} = 5$ and $d \equiv 2,3$ (mod $5$), or $ii)$: $\tilde{g} = 3$ and $d \equiv 1,2$ (mod $3$).  If $(\tilde{g},d) >1$, then 
\[|\Phi_{-1,-1}(\tilde{g},d)| = 0 \Leftrightarrow |\varphi_{d}(\tilde{g})|+(-1)^{r-s}\prod_{i=1}^{s}(q_{i}-1) = 0,\]
which occurs if only if $r = s+1$ and $q_r = 3$. Since $\tilde{g}$ is square-free, this is equivalent to $\tilde{g}|3d \wedge \tilde{g} \nmid d$.\\
\\
2) If $(\tilde{g},d)=1$, then

\[|\Phi_{1,1}(\tilde{g},d)| =0 \Leftrightarrow |\varphi_{d}(\tilde{g})|+(-1)^{r}\bigg(1+\genleg{-d}{\tilde{g}}+\genleg{d}{\tilde{g}}\bigg) = 0,\]
which occurs if and only if either: $i)$: $\tilde{g} = 3$ and $3 \nmid d$, or $ii)$: $\tilde{g}=5$ and $d\equiv 1,4$ (mod 5).  If $(\tilde{g},d)>1$, then 
\[|\Phi_{1,1}(\tilde{g},d)| = 0 \Leftrightarrow |\varphi_{d}(\tilde{g})|+(-1)^{r-s}\prod_{i=1}^{s}(q_{i}-1) = 0,\]
which, as before, occurs if and only if $r = s+1$ and $q_r = 3$.\\
\\
3) and 4) If $(\tilde{g},d)=1$, then 

\[|\Phi_{-1,1}(\tilde{g},d)| = 0 \Leftrightarrow |\varphi_{d}(\tilde{g})|+(-1)^{r}\bigg(1-\genleg{-d}{\tilde{g}}+\genleg{d}{\tilde{g}}\bigg) = 0,\]
which cannot occur.  Similarly, we find that $|\Phi_{1,-1}(\tilde{g},d)| \neq 0$.\\
\\
If $(\tilde{g},d) > 1$, then
\[|\Phi_{-1,1}(\tilde{g},d)| = 0 \Leftrightarrow  |\varphi_{d}(\tilde{g})|-(-1)^{r-s}\prod_{i=1}^{s}(q_{i}-1) = 0,\]
which occurs when $r = s$, i.e. when $\tilde{g}|d$.  Similarly with $\Phi_{1,-1}(\tilde{g},d)$.
\end{proof}
\subsection{{Proof of Theorem \ref{Finite Theorem}}}
\begin{proof}
We wish to identify the instances in which $\textbf{c}_{g}(d)=0$.  When $g \equiv 1(4)$,
\begin{align*}
\textbf{c}_{g}(d) &= \Omega_{d,\tilde{g}}(-1,-1)\\
&=\textbf{e}_{\tilde{g}}(d)\sum_{a \in \Phi_{-1,-1}(\tilde{g}, d)}K_{d,\tilde{g}}(a),
\end{align*}
and therefore since $\textbf{e}_{\tilde{g}}(d), K_{d,\tilde{g}}(a) > 0,$
\[\textbf{c}_{g}(d) = 0 \Leftrightarrow |\Phi_{-1,-1}(\tilde{g}, d)| = 0.\]
Note that by Lemma \ref{Residue Classes}, $\Delta_{d,g}: \pi_{d}\to \{0,1\}$ is supported on the set

\begin{align*}
\{p \in \pi_{d}: \Delta_{d,g}(p) \neq 0 \}&= \{p \in \pi_{d}: (p,g_{s})=(p+d,g_{s})=1 \wedge \delta_{-1,-1}(\tilde{g},d)(p) =1 \}\\
&= \bigcup_{a \in \Phi_{-1,-1}(\tilde{g},d)} \{(p,g_{s})=(p+d,g_{s})=1 \wedge p \equiv a \Mod{\tilde{g}}\}.
\end{align*}
In other words, $\textbf{c}_{g}(d) = 0$ implies that $\Delta_{d,g}$ is supported on the empty set.  Since $p \in \pi_{d,g} \Rightarrow \Delta_{d,g}(p)= 1$, it further follows that $\pi_{d,g}$ is empty whenever $\textbf{c}_{g}(d) = 0$.  The other cases follow similarly, and we conclude that $\textbf{c}_{g}(d) = 0$ if and only if one of the following cases holds:
\begin{itemize}
\item $g_{0} \equiv 1$ (mod $4$) and $|\Phi_{-1,-1}(\tilde{g},d)| = 0$,
\item $g_{0} \equiv 3$ (mod $4$), 
\begin{itemize}
\item $d \equiv 0$ (mod $4$) and $|\Phi_{\pm 1,\pm 1}(\tilde{g},d)| = 0$,
\item $d \equiv 2$ (mod $4$) and $|\Phi_{\pm 1,\mp 1}(\tilde{g},d)|= 0$, 
\end{itemize} 
\item $g_0 \equiv 2$ (mod $4$),
\begin{itemize}
\item $d \equiv 0$ (mod $8$) and $|\Phi_{\pm 1,\pm 1}(\tilde{g},d)| = 0$, 
\item $d \equiv 2,6$ (mod $8$) and $|\Phi_{\pm 1,\pm 1}(\tilde{g},d)|=|\Phi_{\pm 1,\mp 1}(\tilde{g},d)|=0$,
\item $d \equiv 4$ (mod $8$) and $|\Phi_{\pm 1,\mp 1}(\tilde{g},d)| = 0$.
\end{itemize} 
\end{itemize} 
Theorem \ref{Finite Theorem} now follows from Lemma \ref{zerocases}.
\end{proof}

\appendix

\section{Obtaining Numerical Evidence for Conjecture \ref{main conjecture}}

The following Mathematica code is used to generate the set of primes $p \in \pi_{d,g}$ up to height $10^9$, for the particular case $d = 2$ and $g = 10$. To begin, we generate a list of primes $ p \leq 10^9$ (`$\verb|primes|$'), and then selecting those for which $ p-d $ is also prime (`$\verb|primePairs|$'). Next, we select the subset of primes for which $ g $ is a primitive root modulo both $ p $ and $p-d$ (`$\verb|primePairsArtin|$'). The list is subsequently exported into a file. 
\begin{verbatim}
d = 2;
g = 10;
primes = Table[Prime[n], {n, PrimePi[10^9]}];
primePairs = Select[primes, PrimeQ[# - d] &];
primePairsArtin = Select[primePairs, 
  MultiplicativeOrder[g, #] == # - 1 && 
    MultiplicativeOrder[g, # - d] == # - d - 1 &];
Export["primePairsArtin", primePairsArtin, "List"]
\end{verbatim}
To generate the plot found in Figure 1, we first divide the primes in the above list into bins of size $10^5$ (`$\verb|primePairsBins|$'), which we then use to count primes up to size $k \cdot 10^5$, for each $k = 1, 2, \dots$ ((`$\verb|primePairsCounts|$')). Next, we plot pairs of the form $ (k\cdot 10^5, \verb|primePairsCounts[[k]]|),$ where $\verb|primePairsCounts[[k]]|$ denotes the number of primes $p \in \pi_{d,g}$ up to height $k\cdot 10^5$. The plot is then graphed alongside $\verb|const|\cdot\Li_2(x) $, where $\verb|const| = 0.167606 \approx \textbf{c}_{10}(2)\mathbf{A}(2)\mathfrak{S}(2)$ has been manually inserted.
\begin{verbatim}
primePairsArtin = Import["primePairsArtin", "List"];
xMax = 10^9;
step = 10^5;
primePairsBins = BinCounts[primePairsArtin, step];
primePairsCounts = Accumulate[primePairsBins];
kMax = Length[primePairsCounts];
primePairsPlot = Table[{k*step, primePairsCounts[[k]]}, {k, kMax}]
plotPrimes = ListPlot[primePairsPlot, PlotStyle -> Blue];
LogInt2[x_] := NIntegrate[1/Log[s]^2, {s, 2, x}] ;
const = 0.167606;
plotEstimate = 
  Plot[const*LogInt2[x], {x, 2, xMax},
  PlotStyle -> Red, 
  PlotRange -> {{0, xMax}, Automatic}, AxesLabel -> {"x"}, 
  PlotLegends -> LineLegend[{Blue, Red},
    {"pi_{2, 10}(x)", "C(2, 10)Li_2(x)"}]];
P = Show[plotPrimes, plotEstimate]
\end{verbatim}
All the computations were performed on a Windows 10 laptop with an Intel i5-7200U processor and 8 GB RAM.

\end{document}